\documentclass[12pt]{article} % use larger type; default would be 10pt

\usepackage[utf8]{inputenc} % set input encoding (not needed with XeLaTeX)
\usepackage{bbm}

\usepackage{graphicx} % support the \includegraphics command and options

 \usepackage[parfill]{parskip} % Activate to begin paragraphs with an empty line rather than an indent
 
\usepackage{booktabs} % for much better looking tables
\usepackage{array} % for better arrays (eg matrices) in maths
\usepackage{paralist} % very flexible & customisable lists (eg. enumerate/itemize, etc.)
\usepackage{verbatim} % adds environment for commenting out blocks of text & for better verbatim
\usepackage{subfig} % make it possible to include more than one captioned figure/table in a single float
\usepackage{mathrsfs}
\usepackage{amssymb}
\usepackage{amsthm}
\usepackage{amsmath,amsfonts,amssymb,esint}
\usepackage{graphics,color}
\usepackage{enumerate}
\usepackage{mathtools}
\newtheorem{lemma}{Lemma}
\newtheorem{theorem}{Theorem}
\newtheorem{corollary}{Corollary}
\newtheorem{proposition}{Proposition}
\newtheorem{remark}{Remark}

\usepackage{ color, graphicx}
\usepackage{mathrsfs, dsfont}

\usepackage{hyperref}

\newcommand{\RR}{\mathbb{R}}

\newcommand\nc{\newcommand}
\nc\hd{\widehat{D}}
\nc\be{\begin{equation}}
\nc\ee{\end{equation}}
\nc\kp{\kappa}
\nc\8{\infty}

\DeclareMathSymbol{\Gamma}{\mathalpha}{letters}{"00}
\DeclareMathSymbol{\Theta}{\mathalpha}{letters}{"02}
\DeclareMathSymbol{\Lambda}{\mathalpha}{letters}{"03}
\DeclareMathSymbol{\Omega}{\mathalpha}{letters}{"0A}
\definecolor{cr}{rgb}{1,0,0}

\title{Analysis of the (CR) equation in higher dimensions}
\author{T. Buckmaster, P. Germain, Z. Hani, J. Shatah}
\begin{document}
\maketitle

\begin{abstract}
This paper is devoted to the analysis of the continuous resonant (CR)
equation, in dimensions greater than 2. This equation arises as the large
box (or high frequency) limit of the nonlinear Schrodinger equation on the
torus, and was derived in a companion paper by the same authors. We initiate the investigation of the structure of (CR), its local
well-posedness, and the existence of stationary waves.
\end{abstract}
\section{Introduction}

The purpose of this paper is to investigate various properties of the \emph{continuous resonant} (CR) equation, derived in the companion paper \cite{BGHS}. We start with a brief description of this equation.

\subsection{Presentation of the equation}

The~\eqref{CR} equation, derived in~\cite{BGHS} (following~\cite{FGH}) reads
\begin{equation}
\tag{CR} \label{CR} \left\{
\begin{array}{l}
i\partial_t g(t, \xi) = \mathcal {T}(g,g,g)(t, \xi)\qquad (t, \xi) \in \RR\times \RR^d \\
\displaystyle \mathcal T(g,g,g)(\xi)= \iiint_{(\RR^d)^3} g(\xi_1) \overline{g(\xi_2)} g(\xi_3) \delta_{\RR^d}(\xi_1-\xi_2+\xi_3-\xi) \delta_{\RR}(\Omega) \, d\xi_1 d\xi_2 d\xi_3,
\end{array} \right.
\end{equation}
where 
$$
\Omega = |\xi_1|^2 - |\xi_2|^2 + |\xi_3|^2 - |\xi|^2.
$$

This equation describes the effective dynamics of the nonlinear Schr\"odinger equation 
\begin{equation}
\label{NLS} \tag{NLS}
i\partial_t u - \Delta u = |u|^2 u\,,
\end{equation}
set on a box of size $L$ (with periodic boundary conditions) in the limit of large $L$ and small initial data. Equivalently, it also gives the effective dynamics of the same equation posed on the unit torus $\mathbb T^d$ in the high-frequency limit.% of this same equation set on the unit torus. 
Thus, \eqref{CR} is one of the few  simpler models which could help understand of the dynamics of~\eqref{NLS} on compact domains.

In fact, the companion paper \cite{BGHS} derives such continuous resonant effective models for (NLS) with any analytic power nonlinearity (where $|u|^2u$ is replaced by $|u|^{2p}u$ for any positive integer $p$). We shall restrict our analysis in this paper to the cubic case $p=1$, mainly for concreteness purposes. We should note that this equation was derived for $d=2$ in \cite{FGH}, and later analyzed in \cite{GHT1, GHT2}. The pair $(d, p)=(2, 1)$ (where $p$ is the power of the nonlinearity $|u|^{2p}u$) corresponds to the mass-critical case for (NLS). In this case, which also holds for $(d, p)=(1, 2)$, the equation (CR) has very special and surprising properties, like being invariant under the Fourier transform and ``commuting" with the quantum harmonic oscillator \cite{FGH, GHT1}. Moreover, the same equation turns out to give the effective dynamics for various physical systems such as NLS with harmonic trapping \cite{HaniThomann}, the dispersion-managed NLS equation, or lowest-Landau-level dynamics \cite{PG}. %The same type of rich properties should hold for the other mass-critical pair $(d=1, p=2)$. 
It is an interesting question to understand which of the rich properties that hold for (CR) in $d=2$ generalize to higher dimensions. This is one of the purposes of this paper. 

\subsection{Properties} We shall see that (CR) is Hamiltonian with energy functional given by
$$
\mathcal H(g)=\frac{(2\pi)^{d-1}}{2} \int_{\RR_s}\int_{\RR^d_x} \left| e^{is\Delta} \check g\right|^4 dx\, ds.
$$

Moreover, the equation satisfies a handful of symmetries, some are inherited from (NLS), and some are new. These translate into conservation laws. In contrast to the 2D case, the kinetic energy  
$$
\int |\nabla_\xi g(t, \xi)|^2 \, d\xi,
$$
which actually corresponds to the quantity $\int x^2|e^{-it\Delta} u|^2 \, dx$ for \eqref{NLS}, is not conserved for $d\geq 3$. In fact, we shall see that it satisfies the following ``virial-type" identity 
$$
\frac{d}{dt} \| \nabla g \|_{L^2}^2 = 2 (2-d) (2\pi)^{d-1} \iint s |e^{is \Delta} \check g|^4 \,dx\,ds.
$$
This is proved in Section \ref{viriel section}.

In terms of well-posedness, this is directly related to the boundedness properties of the operator $\mathcal T$ in \eqref{CR}. In Proposition \ref{bddness}, we study such properties for a bunch of function spaces. This directly translates into local well-posedness in those function spaces. It should be noted that this includes the function spaces which were used in our companion paper \cite{BGHS} to prove the rigorous approximation result between the (CR) dynamics and that of (NLS).

%Another attractive feature of~\eqref{CR} is its beautiful structure; in particular, as we will see, it is the Hamiltonian flow associated to the  $L^4$ Strichartz norm on $\mathbb{R} \times \mathbb{R}^d$: $\displaystyle \iint_{\mathbb{R} \times \mathbb{R}^d} | e^{is\Delta} f |^4\,dx\,ds$.
%
%Finally,~\eqref{CR} appears as the governing equation for various physical systems such as dispersion-managed equations or lowest-Landau-level dynamics, see~\cite{PG}

%\subsection{Main results}

%The present article investigates the properties of~\eqref{CR} in dimension $d \geq 3$. The case $d=2$ was studied in~\cite{GHT1, GHT2}, but its structure is very different from higher dimensions. This is related to the $L^2$-criticality of~\eqref{NLS} in dimension 2, or to put it differently, to the pseudo-conformal symmetry which only holds for $d=2$. The same holds for the quintic nonlinearity in $d=1$.

\subsubsection{Stationary solutions} It was noticed in \cite{FGH, GHT1}, equation (CR) enjoys a wealth of stationary solutions when $d=2$. This motivated us to study stationary solutions for this equation in higher dimensions. Here, we construct them by variational methods as maximizers of the Hamiltonian functional with fixed mass and/or weighted $L^2$ norm. We should point out that we only took the first steps in this study and many remaining interesting questions remain open for investigation. 

\subsection{Notations}

The Fourier transform and the Fourier transform of a function on $\mathbb{R}^d$ are given by, respectively,
\begin{align*}
& \mathcal{F} f(\xi) = \widehat{f}(\xi) = \frac{1}{(2\pi)^{d/2}} \int_{\mathbb{R}^d} e^{-ix\xi} f(x) \, dx \\
& \mathcal{F}^{-1} f(x) = \check{f}(x) =  \frac{1}{(2\pi)^{d/2}} \int_{\mathbb{R}^d} e^{ix\xi} f(\xi) \, d\xi.
\end{align*}
The function spaces $L^{p,s} = L^{p,s}(\mathbb{R}^d)$, $\dot L^{p,s} = \dot L^{p,s}(\mathbb{R}^d)$, $W^{p,s} = W^{p,s}(\mathbb{R}^d)$, $X^{\ell, N}=X^{\ell, N}(\mathbb{R}^d)$ are given by their norms
\begin{align*}
& \| f \|_{L^{p,s}} = \| \langle x \rangle^s f(x) \|_{L^p} \\
& \| f \|_{\dot L^{p,s}} = \| |x|^s f(x) \|_{L^p} \\
& \| f \|_{W^{p,s}} = \| \langle D \rangle^s f \|_{L^p}\\
&\|f\|_{X^{\ell,N}(\mathbb{R}^d)}=\sum_{0\leq |\alpha|\leq N} \|\nabla^\alpha f\|_{L^{\infty, \ell}}.
\end{align*}

\section{Structure of the equation} 

\subsection{The Hamiltonian}

In this subsection, we remain at a formal level, and discuss the structure of~\eqref{CR}. Notice first that it is Hamiltonian: it can be written
$$
i \dot g = \frac{\partial \mathcal H(g)}{\partial \overline g}
$$
with
$$ 
\mathcal H(g)= \frac{1}{2}\iiiint_{(\mathbb{R}^d)^4} g(\xi_1) \overline{g(\xi_2)} g(\xi_3) \overline{g(\xi_4)} \delta_{\RR^d}(\xi_1-\xi_2+\xi_3-\xi_4) \delta_{\RR}(\Omega) \, d\xi_1 d\xi_2 d\xi_3d\xi_4.
$$
We will also use the polarized version
$$
\mathcal H(g_1,g_2,g_3,g_4) =\frac{1}{2}\iiiint_{(\mathbb{R}^d)^4} g_1(\xi_1) \overline{g_2(\xi_2)} g_3(\xi_3) \overline{g_4(\xi_4)} \delta_{\RR^d}(\xi_1-\xi_2+\xi_3-\xi_4) \delta_{\RR}(\Omega) \, d\xi_1 d\xi_2 d\xi_3d\xi_4.
$$
First notice that, viewed in physical space, the Hamiltonian is nothing but the $L^4$ Strichartz norm.

\begin{proposition}
\label{physicalspace}
The Hamiltonian and its polarized version can be written
\begin{align*}
&\mathcal H(g)= \frac{(2\pi)^{d-1}}{2} \int_{\RR_s}\int_{\RR^d_x} \left| e^{is\Delta} \check g\right|^4 dx\, ds \\
& \mathcal H(g_1, g_2,g_3, g_4) = \frac{(2\pi)^{d-1}}{2} \int_{\RR_s}\int_{\RR^d_x} e^{is\Delta} \check g_1(x)\,\overline{e^{is\Delta}\check g_{2}(x)}\, e^{is\Delta} \check g_3(x)\,\overline{e^{is\Delta}\check g_{4}(x)}\, dx\, ds.
\end{align*}
As for the trilinear operator $\mathcal T$, it reads
$$
\mathcal T(g_1, g_2, g_3)=(2\pi)^{d-1} \mathcal F \int_{\RR_s}e^{-is\Delta}\left( e^{is\Delta} \check g_1(x)\overline{e^{is\Delta}\check g_2(x)}e^{is\Delta} \check g_3(x)\right)\,ds.
$$
\end{proposition}

\begin{proof}
It suffices to prove the identity for the polarized Hamiltonian function. Using that $(2\pi)^{-1}\int_{\RR}e^{ix\xi}d\xi= \delta_{x=0}$, we obtain that 
\begin{align*}
&\int_{\RR_s}\int_{\RR^d_x} e^{is\Delta} \check g_1(x)\overline{e^{is\Delta}\check g_2(x)}e^{is\Delta} \check  g_3(x)\overline{e^{is\Delta}\check  g_4(x)}dx \,ds\\
& =(2\pi)^{-2d}\int_{\RR_s}\int_{\RR^d_x} \int_{\RR^{4n}} e^{ix(\eta_1-\eta_2+\eta_3-\eta_4)}e^{-is(\eta_1^2-\eta_2^2+\eta_3^2-\eta_4^2)} g_1(\eta_1) \overline{ g_2}(\eta_2) g_3(\eta_3) \overline{ g_4}(\eta_4) d\eta_1 \ldots d\eta_4 \,dx \,ds\\
& =(2\pi)^{-d}\int_{\RR_s} \int_{\RR^{4d}} e^{-is(\eta_1^2-\eta_2^2+\eta_3^2-\eta_4^2)}g_1(\eta_1) \overline{ g_2}(\eta_2) g_3(\eta_3) \overline{ g_4}(\eta_4)\delta_{\RR^d}(\eta_1-\eta_2+\eta_3-\eta_4) d\eta_1 \ldots d\eta_4 \,ds.\\
& =(2\pi)^{-d+1} \int_{\RR^{4d}} g_1(\eta_1) \overline{g_2}(\eta_2) g_3(\eta_3) \overline{g_4}(\eta_4)\delta_{\RR^d}(\eta_1-\eta_2+\eta_3-\eta_4) \delta_{\RR}(\eta_1^2-\eta_2^2+\eta_3^2-\eta_4^2) d\eta_1 \ldots d\eta_4.
\end{align*}
\end{proof}

\subsection{Symmetries and conserved quantities}

\begin{proposition}[Symmetries]
\label{symm}
The following symmetries leave the Hamiltonian $\mathcal{H}$ invariant:
\begin{itemize}
\item[(i)] Phase rotation $g \mapsto e^{i\theta} g$ for all $\theta \in \mathbb{R}$.
\item[(ii)] Translation: $g \mapsto g(\,\cdot\,+\xi_0)$ for any $\xi_0 \in \mathbb{R}$.
\item[(iii)] Modulation: $g \mapsto e^{i \xi \cdot x_0} g$ for any $x_0 \in \mathbb{R}^d$.
\item[(iv)] Quadratic modulation: $g \mapsto e^{i\tau |\xi|^2} g$ for any $\tau \in \mathbb{R}$.
\item[(v)] Rotation: $g \mapsto g(O \cdot)$ for any $O$ in the orthogonal group $O(d)$.
\item[(vi)] Scaling: $g \mapsto \mu^{\frac{3d-2}{4}} g(\mu \,\cdot\,)$ for any $\mu > 0$.
\end{itemize}
\end{proposition}

The first five symmetries are canonical transformations (except in dimension 2, where the scaling is also canonical) which, by Noether's theorem, are associated to conserved quantities.

\begin{proposition}[Conserved quantities]\label{ConsQ}
The following quantities are conserved by the flow of~\eqref{CR}:
\begin{itemize}
\item[(i)] Mass: $\int |g(\xi)|^2\,d\xi$.
\item[(ii)] Position: $\int x \left| \check{g}(x) \right|^2\,dx$.
\item[(iii)] Momentum: $\int \xi |g(\xi)|^2 \,d\xi$.
\item[(iv)] Kinetic energy: $\int \xi^2 |g(\xi)|^2 \,d\xi$.
\item[(v)] Angular momentum: $\int (\xi_i \partial_{\xi_j} - \xi_j \partial_{\xi_i}) g(\xi) \overline{g(\xi)} \,d\xi$ for $i \neq j$.
\end{itemize}
\end{proposition}

Notice further that the first four symmetries of the Hamiltonian naturally translate into symmetries of the set of solutions of the equation. Furthermore, the following transformation leaves the set of solutions of~\eqref{CR} invariant:
$$
g(t,\xi) \mapsto \lambda g(\lambda^2 \mu^{2-2d} t,\mu x), \qquad \mbox{for $\mu,\lambda >0$}.
$$

\begin{remark}
In the 2D cubic case and the 1D quintic case, the~\eqref{CR} equation enjoys a much bigger group of symmetries and more conserved quantities due the mass critical nature of the nonlinearity. We refer to \cite{FGH} for details.
\end{remark}

\subsection{Virial formula}\label{viriel section}

The virial formula is a classical tool in the study of the nonlinear Schr\"odinger equation. It has an analog for the~\eqref{CR} equation.

\begin{proposition} If $g$ is a solution of the~\eqref{CR} equation,
$$
\frac{d}{dt} \| \nabla g \|_{L^2}^2 = 2 (2-d) (2\pi)^{d-1} \iint s |e^{is \Delta} \check g|^4 \,dx\,ds.
$$
\end{proposition}

\begin{remark}
The above theorem shows a sharp distinction between the special dimension $d=2$ and higher dimensions. Indeed, this conservation was observed in \cite{FGH} for $d=2$, but what seems interesting is that the $\dot H^1$ norm is \emph{not} conserved in higher dimensions. 
\end{remark}

\begin{proof}
By Proposition~\eqref{physicalspace}, and using that $\mathcal{F}$ is an isometry on $L^2$, 
\begin{align*}
\frac{1}{2} \frac{d}{dt} \int |\nabla g|^2\,d\xi & = (2\pi)^{d-1} \mathfrak{Re} \iint i \mathcal{F} e^{-is\Delta} \left( e^{is\Delta} \check g \, \overline{e^{is\Delta}\check g}\,e^{is\Delta} \check g\right) \Delta \overline g \,d\xi \,ds \\
& = (2\pi)^{d-1} \mathfrak{Im} \iint e^{is\Delta} \check g\,\overline{e^{is\Delta}\check g}\,e^{is\Delta} \check g \,\overline{e^{is\Delta} x^2 \check g}\,dx\,ds.
\end{align*}
Using the identity
$$
e^{is\Delta} x^2 \phi = x^2 e^{is\Delta} \phi + 2isd e^{is\Delta} \phi + 4isx \cdot \nabla e^{is\Delta} \phi - 4s^2 \Delta e^{is\Delta} \phi,
$$
the above becomes
\begin{align*}
& \frac{1}{2} \frac{d}{dt} \int |\nabla g|^2\,dx \\
&\qquad  = (2\pi)^{d-1} \mathfrak{Im} \iint e^{is\Delta} \check g\,\overline{e^{is\Delta}\check g}\,e^{is\Delta} \check g \, \overline{\left[2isd e^{is\Delta} \check g + 4isx \cdot \nabla e^{is\Delta} \check g - 4s^2 \Delta e^{is\Delta} \check g \right]}\,dx\,ds \\
& \qquad = I + II + III
\end{align*}
Through integrations by parts, it is easy to see that
\begin{align*}
& I = - (2\pi)^{d-1} 2d \iint s |e^{is \Delta} \check g|^4 \,dx\,ds \\
& II = (2\pi)^{d-1} d \iint s |e^{is \Delta} \check g|^4 \,dx\,ds \\
& III = (2\pi)^{d-1} 2 \iint s |e^{is \Delta} \check g|^4 \,dx\,ds,
\end{align*}
from which the desired result follows.
\end{proof}

\section{Local and global well-posedness}
The first step is to establish boundedness properties of the trilinear operator $\mathcal{T}$.

\begin{proposition}\label{bddness}
The trilinear operator $\mathcal{T}$ is bounded from $X \times X \times X$ to $X$ for the following Banach spaces $X$:
\begin{itemize}
\item[(i)] $X = \dot L^{2,\frac{d-2}{2}}$
\item[(ii)] $X = L^{2,s}$ for $s \geq \frac{d-2}{2}$.
\item[(iii)] $X = L^{\infty,s}$, for $s>d-1$.
\item[(iv)] $X = L^{p,s}$ for $p \geq 2$ and $s > d-1-\frac{d}{p}$.
\item[(v)] $X=X^{\sigma, N}$ for any $\sigma> d-1$ and $N \geq 0$.
\end{itemize}
\end{proposition}

%\comment{(PG) It would be nice to understand the scale-invariant space $X = \dot L^{\infty,d-1}$... but the computations seem hard!}

The borderline spaces for well-posedness in the above proposition are $\dot L^{p,s}$, with $s = d-1-\frac{d}{p}$. They share the same scaling, and are also scale-invariant for the cubic NLS $i\partial_t u - \Delta u = |u|^2 u$ set in $\mathbb{R}^d$ (when viewed as spaces for $\widehat{u}$).

\begin{proof} The assertions $(i)$ and $(ii)$ follow from applying successively Proposition~\ref{physicalspace}, the Minkowski inequality, the fractional Leibniz rule, Strichartz' inequality, and Sobolev embedding (cf.\ Appendix A, \cite{TT}): indeed,
\begin{align*}
\|\mathcal T (f_1,f_2, f_3)\|_{L^{2,s}}\leq &(2\pi)^{d-1}\left \|e^{it\Delta} \check f_1(x)\overline{e^{it\Delta}\check f_2(x)}e^{it\Delta} \check f_3(x)ds\right\|_{L_t^1 H^s}\\
\lesssim& \sum_{j=1}^3 \|e^{it\Delta} \check f_j\|_{L_t^{\frac{2(d+2)}{d}} W^{s,\frac{2(d+2)}{d}}} \prod_{k\neq j}\|e^{it\Delta} \check f_j\|_{L_t^{\frac{4(d+2)}{d+4}} L_x^{2(d+2)}} \lesssim \prod_{j=1}^3\|\check f_j\|_{H^{s}}.
\end{align*}
provided $s\geq \frac{d-2}{2}$.

The assertion $(iv)$ is obtained by interpolating between $(ii)$ and $(iii)$; therefore, it only remains to prove $(iii)$. It is equivalent to proving that, for $s>d-1$,
$$
\sup_\xi \iiint_{(\RR^d)^3} \frac{\langle \xi \rangle^s}{\langle \xi_1 \rangle^s \langle \xi_2 \rangle^s \langle \xi_3 \rangle^s} \delta_{\RR^d}(\xi_1-\xi_2+\xi_3-\xi) \delta_{\RR}(\Omega) \, d\xi_1 d\xi_2 d\xi_3 < \infty.
$$
Observe that the first $\delta$ in the above integrand imposes $\xi_2 = \xi_1 + \xi_3 - \xi$. Switching to the integration variables $y = \frac{\xi_1 + \xi_3 - 2\xi}{2}$ and $z = \frac{\xi_1 - \xi_3}{2}$, the above becomes
$$
\sup_\xi \iint_{(\RR^d)^2} \frac{\langle \xi \rangle^s}{\langle \xi + y +z \rangle^s \langle \xi + y -z \rangle^s \langle \xi + 2y \rangle^s} \delta(y^2 - z^2) \,dy\,dz < \infty.
$$
Since the $\delta$ function in the above integrand forces $|y| = |z|$, either $\langle \xi \rangle \lesssim \langle \xi + y +z \rangle$ or $\langle \xi \rangle \lesssim \langle \xi + y -z \rangle$. Assuming without loss of generality that the former occurs, it suffices to show that
$$
\sup_\xi \iint_{(\RR^d)^2} \frac{1}{ \langle \xi + y -z \rangle^s \langle \xi + 2y \rangle^s} \delta(y^2 - z^2) \,dy\,dz < \infty.
$$
We will now rely on the simple estimate: if $A \in \mathbb{R}^{d}$, $r>0$, and $s > d-1$, 
$$
\int_{\mathbb{S}^{d-1}} \frac{1}{\langle A + r \omega \rangle^s} \,d\omega \lesssim \langle |A| - r \rangle^{d-1-s} \langle r \rangle^{1-d}
$$
(in order to check this estimate, observe first that $\int_{\mathbb{S}^{d-1}} \frac{1}{\langle A + r \omega \rangle^s} \,d\omega \sim \int_0^\pi \frac{(\sin \phi)^{d-2}}{(|A|^2 + r^2 - 2|A|r \cos \phi + 1)^{s/2}} \, d\phi$, and then that, close to $\phi = 0$, this becomes $\sim \int_0^1 \frac{\phi^{d-2}}{(\langle |A| - r \rangle^2 + |A| r \phi^2 )^{s/2}} \, d\phi $).
Changing the integration variables to $y = r \omega$ and $z = r \phi$, and using the inequality above twice gives the desired result:
\begin{align*}
& \sup_\xi \iint_{(\RR^d)^2} \frac{1}{ \langle \xi + y -z \rangle^s \langle \xi + 2y \rangle^s} \delta(y^2 - z^2) \,dy\,dz \\
& \qquad \qquad \qquad =  \sup_\xi \int_0^\infty \int_{(\mathbb{S}^{d-1})^2} \frac{1}{\langle \xi + r \omega - r \phi \rangle^s \langle \xi + 2 r \omega  \rangle^s} r^{2d-3} \, d\omega \, d\phi\,dr \\
& \qquad \qquad \qquad \lesssim  \sup_\xi \int_0^\infty \int_{\mathbb{S}^{d-1}}  \frac{1}{ \langle \xi + 2 r \omega  \rangle^s} r^{d-2} \,d\omega \, dr \\
& \qquad \qquad \qquad \lesssim  \sup_\xi \int_0^\infty \frac{1}{\langle r \rangle \langle |\xi| - 2r \rangle^{s-d+1}}\,dr < \infty.
\end{align*}

Finally, part $(v)$ follows from part $(iii)$ and Leibniz rule. 
\end{proof}

Local and global regularity properties of (CR) follow easily:

\begin{corollary}
\begin{enumerate}
\item[(i)] Local well-posedness: For $X$  any of the spaces given in Proposition \ref{bddness}, the 
Cauchy problem~\eqref{CR} is locally well-posed in $X$.
More precisely, for any $g_0$ in $X$, there exists a time $T>0$, and  a solution in $\mathcal{C}^\infty ([0,T],X)$,
which is unique in $L^\infty ([0,T],X)$, and depends continuously on $g_0$ in this topology. 
\item[(ii)] Global well-posedness for finite mass: if $d=2$ and $g_0 \in L^{2}$, the local solution can be prolonged into a global one.
More precisely: there exists a unique solution $ g \in \mathcal{C}^\infty ([0,\infty),L^{2})$ which, for any $T$, is unique
in $L^\infty ([0,T],L^{2})$, and depends continuously on $g_0$ in this topology. 
\item[(iii)] Global well-posedness for finite kinetic energy: if $d = 2,3,4$ and $g_0 \in L^{2,1}$, the local solution can be prolonged into a global one.
\item[(iv)] Propagation of regularity: assume $g_0 \in \dot L^{2, \frac{d-2}{2}}$, and let $g$ be the solution given in (i).
If in addition $g_0 \in L^{2, \sigma}$ for $\sigma  \geq \frac{d-2}{2}$ then $g \in \mathcal{C}^\infty ([0,T), L^{2,\sigma})$. 
\end{enumerate}
\end{corollary}

\begin{proof}
The proof of $(i)$ is immediate since~\eqref{CR} is an ODE in $X$. Combining this local well-posedness result with conservation laws classically gives global well-posedness, leading to $(ii)$ and $(iii)$. Finally, $(iv)$ is classical.
\end{proof}

\section{Stationary waves}

\subsection{Basic properties}

We will discuss here the existence of solutions of the type
$$
g(t,\xi) = e^{-i (\mu + \lambda |\xi|^2 + \nu \cdot \xi) t } \psi,
$$
where $\lambda, \mu \in \mathbb{R}$, and $\nu \in \mathbb{R}^d$. For $g$ to solve~\eqref{CR}, it suffices that $\psi$ solves
$$
(\lambda |\xi|^2 + \mu + \nu \cdot \xi) \psi = \mathcal{T}(\psi,\psi,\psi).
$$
Notice that $g$ defined above oscillates in Fourier space, but it actually travels in physical space, as can be seen by taking its inverse Fourier transform:
$$
\check g(t,x) = e^{-i (\mu - \lambda \Delta) t} \check \psi(x - \nu t).
$$
The conservation of position (part $(ii)$ of Proposition \ref{ConsQ}) gives a restriction on the relation between $\nu$ and $\lambda$. Indeed, using the identity $[x, e^{it\Delta}]=-2it\nabla e^{it\Delta}$, one obtains that
$$
\int x|\check g|^2\, dx=\int x |\check \psi|^2 dx+t\left(\nu M(\check \psi) - 2\lambda P(\widehat \psi) \right)
$$
where $M(\check \psi)= \int |\check \psi|^2 dx$ and $P(\check \psi)=i \int \nabla \check \psi \overline{\check \psi} \, dx \in \mathbb{R}$. Since $\int x |\check g|^2 dx$ is conserved, one must have that 
$$
\nu= \frac{2 P(\check \psi)}{M(\check \psi)} \lambda. 
$$

By invariance of $\mathcal{T}$ under translations, we can define 
$$
\phi(\xi)=\psi(\xi -\frac{\nu}{2\lambda})
$$
which solves
\begin{equation}
\label{eqphi}
(\lambda |\xi|^2 + \mu) \phi = \mathcal{T}(\phi,\phi,\phi);
\end{equation}
we will from now on focus on this equation. 

\medskip

\begin{lemma}[Energy and Pohozaev identities]
\label{lemmep}
Assume that $\phi$ solves~\eqref{eqphi}, and that $\mathcal{H}(\phi) < \infty$. If furthermore $\phi \in L^{2,1}$, it satisfies the energy identity
\begin{equation}
\label{energyidentity}
\lambda \| \xi \phi \|_{L^2}^2 + \mu \| \phi \|_{L^2}^2 = \mathcal{H}(\phi).
\end{equation}
If furthermore $\phi \in L^2$ and $\xi \nabla \phi \in L^2$, it satisfies the Pohozaev identity
\begin{equation}
\label{pohozaevidentity}
\lambda \left( \frac{d}{2} - 1 \right) \| \xi \phi \|_{L^2}^2 + \mu \frac{d}{2} \| \phi \|_{L^2}^2 = \left( \frac{1}{2} + \frac{d}{4} \right) \mathcal{H}(\phi).
\end{equation}
\end{lemma}
\begin{proof} The energy identity follows immediately from taking the $L^2$ scalar product of~\eqref{eqphi} with $\phi$. As for the Pohozaev identity, take first the inverse Fourier transform to obtain
$$
(-\lambda \Delta + \mu) \check \phi = (2 \pi)^{d-1} \int e^{-is\Delta} \left( e^{is\Delta} \check \phi \, \overline{e^{is\Delta} \check \phi} \, e^{is\Delta} \check \phi \right) \,ds.
$$
Pairing this identity with $x \cdot \nabla \phi$ and subsequently taking the real part gives three terms, which after integrations by parts read
\begin{align*}
& \mathfrak{Re} -\lambda \int \Delta \check \phi \, \overline{x \cdot \nabla \check \phi} \, dx = \lambda \left( 1 - \frac{d}{2} \right) \int |\nabla \check \phi|^2\, dx \\
& \mathfrak{Re} \, \mu \int \check \phi \, \overline{x \cdot \nabla \check \phi} \, dx = -\mu \frac{d}{2} \int |\check \phi|^2 \,dx \\
& \mathfrak{Re} \,(2 \pi)^{d-1} \iint e^{-is\Delta} \left( e^{is\Delta} \check \phi \, \overline{e^{is\Delta} \check \phi} \, e^{is\Delta} \check \phi \right) \,ds \, \overline{x \cdot \nabla \check \phi} \,dx = - (2 \pi)^{d-1} \left( \frac{1}{2} + \frac{d}{4} \right) \iint |e^{it\Delta} \check \phi|^4 \, dx\,ds.
\end{align*}
This gives the desired identity!
\end{proof}

Linear combinations of~\eqref{energyidentity} and~\eqref{pohozaevidentity} reveal that, under the hypotheses of Lemma~\ref{lemmep},
\begin{align*}
& \lambda \| \xi \phi \|_{L^2}^2 = \frac{d-2}{4} \mathcal{H} (\phi) \\
& \mu \| \phi \|_{L^2}^2 = \frac{6-d}{4} \mathcal{H}(\phi).
\end{align*}
In particular, still under the hypotheses of Lemma~\ref{lemmep}, necessary conditions for~\eqref{eqphi} to admit a solution are
\begin{itemize}
\item If $d=2$, $\lambda = 0$ and $\mu > 0$.
\item If $3 \leq d \leq 5$, $\lambda > 0$ and $\mu >0$.
\item If $d=6$, $\lambda > 0$ and $\mu = 0$.
\item If $d \geq 7$, $\lambda > 0$ and $\mu < 0$.
\end{itemize}
If $2 \leq d \leq 6$, the existence of a solution is ensured by the following theorem, but in less smooth classes of solutions.

\begin{remark}
We construct solutions to those equations below, but in slightly less smooth classes. It would be interesting to prove regularity of such solutions to bridge the gap.
\end{remark}

\subsection{The variational problem}

\begin{theorem}
\label{theoremvariational}
The following variational problems admit nonzero maximizers:
\begin{itemize}
\item $\displaystyle \sup_{\| g \|_{L^2} = 1} \mathcal{H} (g)$ if $d = 2$.
\item $\displaystyle \sup_{\| g \|_{L^2}^2 + \| \xi g\|_{L^2}^2 = 1} \mathcal{H} (g)$ if $3 \leq d \leq 5$.
\item $\displaystyle \sup_{\| \xi g \|_{L^2} = 1} \mathcal{H} (g)$ if $d=6$.
\end{itemize}
Furthermore, maximizing sequences are compact modulo the symmetries of the equation. In dimension 2, maximizers are given by Gaussians.
\end{theorem}

This variational problem makes sense due to the Strichartz estimates
\begin{align*}
&\mathcal{H}(g) = \| e^{it\Delta} \check{g} \|_{L^4_{t,x}} \lesssim \| g \|_{L^2} \qquad \mbox{if $d=2$} \\
&\mathcal{H}(g) = \| e^{it\Delta} \check{g} \|_{L^4_{t,x}} \lesssim \| g \|_{L^{2,1}} \qquad \mbox{if $3 \leq d \leq 5$} \\
&\mathcal{H}(g) = \| e^{it\Delta} \check{g} \|_{L^4_{t,x}} \lesssim \| g \|_{\dot L^{2,1}} \qquad \mbox{if $d=6$}.
\end{align*}
Therefore, these variational problems belong to the class which arises from Fourier restriction functionals. The cases $d=2$ and $d=6$, as extremizers of the Strichartz~\cite{DF,HZ,FGH} and Sobolev-Strichartz inequalities~\cite{SS} are already known; as we will see, the case $3 \leq d \leq 5$, which is somewhat simpler since it is not critical, will follow from an application of the concentration-compactness technology (as in~\cite{K, DHR}). 

%\comment{Pierre $\to$ Zaher: I am not sure this is the right reference, what do you think? ZH: I think the Keraani reference is enough; as \cite{DHR} reply on it if I'm not mistaken. Perhaps we can add a reference to Gerard and Bahour-Gerard, but that's not absolutely necessary.}

Of great interest is the question of the uniqueness (modulo symmetries) of extremizers, but we can only establish this in a handful of cases, such as the case $d=2$ of the above theorem.

The Euler-Lagrange equations satisfied by the maximizers of these variational problems read
\begin{itemize}
\item $\lambda g = \mathcal{T}(g,g,g)$ (for a Lagrange multiplier $\lambda$) if $d=2$.
\item $\lambda [ g + |\xi|^2 g] = \mathcal{T}(g,g,g)$ (for a Lagrange multiplier $\lambda$) if $3 \leq d \leq 5$.
\item $\lambda |\xi|^2 g = \mathcal{T}(g,g,g)$ (for a Lagrange multiplier $\lambda$) if $d=6$.
\end{itemize}
These three equations should be understood as equalities in $L^2$, $H^{-1}$, and $\dot H^{-1}$ respectively. Since the maximizers are nonzero, testing the above equations against $g$ reveals that $\lambda>0$; up to scaling it can be taken equal to 1. Therefore, we find nonzero solutions of
\begin{itemize}
\item $g = \mathcal{T}(g,g,g)$  if $d=2$.
\item $ g + |\xi|^2 g = \mathcal{T}(g,g,g)$ if $3 \leq d \leq 5$.
\item $|\xi|^2 g = \mathcal{T}(g,g,g)$ if $d=6$.
\end{itemize}
We now turn to the proof of the theorem.

\begin{proof} The case $d=2$ was proved in~\cite{HZ,DF}, see also~\cite{FGH}; the case $d=6$ can be found in \cite{SS}. There remains $d = 3,4,5$, to which we now turn. It is more convenient to take the Fourier transform of the above and consider the variational problem
\begin{equation}
\label{var}
\sup_{\| f \|_{L^2}^2 + \| \nabla f \|_{L^2}^2 = 1} \left\| e^{it \Delta} f \right\|_{L^4_{t,x}}^4.
\end{equation}

\noindent \underline{Step 1: scaling of the problem.} For $A>0$, let
$$
I (A) = \sup_{\| f \|_{L^2}^2 + \| \nabla f \|_{L^2}^2 = A} \left\| e^{it \Delta} f \right\|_{L^4_{t,x}}^4.
$$
It is clear that
$$
I(A) = A^2 I(1).
$$

\noindent \underline{Step 2: profile expansion.}
As is well-known, the lack of compactness of this variational problem can be overcome through a profile expansion. The one needed for this particular variational problem is very similar to the one in~\cite{DHR}, and it can be proved along the same lines. The statement is as follows: consider $(f_n)$ a bounded sequence in $H^1$. Then there exists a subsequence, also denoted $(f_n)$, a second sequence $(\psi^j)$, and doubly indexed subsequences $(t_n^j)$ and $(x_n^j)$ giving for any $J$ the decomposition
$$
f_n (x) = \sum_{j=1}^J e^{it_n^j \Delta} f(x + x^j_n) + r^J_n.
$$
Furthermore
\begin{itemize}
\item The expansion is orthogonal in the Strichartz norm: 
$$
\lim_{J \to \infty} \operatorname{limsup}_{n \to \infty} \left[ \| e^{it\Delta} f_n \|_{L^4_{t,x}}^4 - \sum_{j=1}^J \| e^{it\Delta} \psi_j \|_{L^4_{t,x}}^4 \right] = 0.
$$
\item The expansion is orthogonal in $L^2$: for any $J$,
$$ 
\lim_{n \to \infty} \left[ \| f_n \|_{L^2}^2 - \sum_{j=1}^J \| \psi_j \|_{L^2}^2 - \| r^J_n \|_{L^2}^2 \right] = 0.
$$
\item The expansion is orthogonal in $\dot H^1$: for any $J$,
$$
\lim_{n \to \infty} \left[ \| \nabla f_n \|_{L^2}^2 - \sum_{j=1}^J \|\nabla \psi_j \|_{L^2}^2 - \| \nabla r^J_n \|_{L^2}^2 \right] = 0.
$$
\end{itemize}

\noindent \underline{Step 3: maximizing sequence.} Pick $(f_n)$ a maximizing sequence for the variational problem~\eqref{var}. Then, due to the orthogonality property in the Strichartz norm,
$$
I(1) = \lim_{n \to \infty} \| e^{it\Delta} f_n \|_{L^4_{t,x}}^4 = \sum_{j=1}^\infty \| e^{it\Delta} \psi_j \|_{L^4_{t,x}}^4.
$$
By the scaling property of the variational problem, the above implies that
$$
I(1) \leq \sum_{j=1}^\infty I( \| \psi_j \|_{L^2}^2 +  \| \nabla \psi_j \|_{L^2}^2) \leq I(1) \sum_{j=1}^\infty [ \| \psi_j \|_{L^2}^2 +  \| \nabla \psi_j \|_{L^2}^2]^2,
$$
which implies in turn that
$$
1 \leq \sum_{j=1}^\infty [ \| \psi_j \|_{L^2}^2 +  \| \nabla \psi_j \|_{L^2}^2]^2 \qquad \mbox{while} \qquad  \sum_{j=1}^\infty  \| \psi_j \|_{L^2}^2 +  \| \nabla \psi_j \|_{L^2}^2 \leq 1.
$$
This is only possible if only one of the $(\psi_j)$ is nonzero, say $\psi_1$. In other words, the maximizing sequence is compact (modulo symmetries), and $\psi_1$ is the desired maximizer.
\end{proof}

\subsection{Decay for the Euler-Lagrange problem}

\begin{proposition}[Fourier decay of g]
\begin{itemize}
\item[(i)] If $d=2$, a solution $g \in L^2$ of $g = \mathcal{T}(g,g,g)$ belongs to $L^{2,s}$ for all $s>0$.
\item[(ii)] If $3 \leq d \leq 5$, a solution $g \in L^{2,1}$ of $g + |\xi|^2 g = \mathcal{T}(g,g,g)$ belongs to $L^{2,s}$ for all $s>0$.
\end{itemize}
\end{proposition}

%\comment{(Pierre) Can we say anything if $d=6$?}

\begin{proof} The case $d=2$ was treated in~\cite{FGH}. In fact, much more can be said there and one can prove exponential decay and analyticity of stationary solutions \cite{GHT1}.

If $d=3$, we know that $g \in L^{2,1}$. By Proposition~\ref{bddness}, this implies that $\mathcal{T}(g,g,g) \in L^{2,1}$. The equation satisfied by $g$ implies that $g \in L^{2,3}$, and iterating this argument gives the desired result. 

The case $d=4$ is similar.

If $d=5$, we will show by duality that $\| \mathcal{T}(g,g,g) \|_{L^2} \lesssim \| g \|_{L^{2,1}}^3$, from which the iterative process can be started. Indeed, using H\"older's inequality, the Strichartz estimate and Sobolev embedding, 
\begin{align*}
\left| \langle \mathcal{T}(g,g,g) \,,\, F \rangle \right| & = (2\pi)^{d-1} \left| \int_{\mathbb{R}}  \int_{\mathbb{R}^d} e^{is\Delta} F | e^{is\Delta} g |^2 e^{is\Delta} g \, dx \,ds \right| \\
& \lesssim \| e^{it\Delta} \check{g} \|_{L^4_t L^5_x}^3 \| e^{it\Delta} F \|_{L^4_t L^{5/2}_x}  \\
& \lesssim \| \check{g} \|_{H^1} \| F \|_{L^2} = \| g \|_{L^{2,1}}^3 \| F \|_{L^2}.
\end{align*}
\end{proof}

Although we don't pursue it here, we remark that the issue of regularity of those solutions is very interesting.

%\comment{(Pierre) Finally, there remains the question of regularity in Fourier / decay in physical space. One approach would be to compute the kernel in physical space, but I got stuck. Here is why:
%\begin{align*}
%& \mathcal{F} \mathcal{T}(g,g,g) \\
%& = \int e^{i(x\xi - x_1 \xi_1 + x_2 \xi_2 - x_3 \xi_3)} g(x_1) \overline{g(x_2)} g(x_3) \delta(\xi - \xi_1 + \xi_2 - \xi_3)  \delta(\xi^2 - \xi_1^2 + \xi_2^2 - \xi_3^2) \, d\xi_1 \,d\xi_2\, d\xi_3 \,d\xi\, dx_1 \,dx_2 \,dx_3
%\end{align*}
%so that the kernel reads, after changing variables $\eta_i = \xi_i - \xi$,
%\begin{align*}
%& K(x,x_1,x_2,x_3) \\
%& =  \int e^{i(x\xi - x_1 \xi_1 + x_2 \xi_2 - x_3 \xi_3)}\delta(\xi - \xi_1 + \xi_2 - \xi_3)  \delta(\xi^2 - \xi_1^2 + \xi_2^2 - \xi_3^2) \, d\xi_1 \,d\xi_2\, d\xi_3 \,d\xi\\
%& = \int e^{i \xi (x-x_1+x_2 - x_3)} e^{i(-x_1 \eta_1 + x_2 \eta_2 - x_3 \eta_3)} \delta(\eta - \eta_1 + \eta_2 - \eta_3)  \delta(\eta^2 - \eta_1^2 + \eta_2^2 - \eta_3^2) \, d\eta_1 \,d\eta_2\, d\eta_3 \,d\xi \\
%&= \delta(x-x_1 + x_2 - x_3) \int e^{i(\eta_1 (x_2 - x_1) + \eta_3 (x_2 - x_3))} \delta(2\eta_1 \cdot \eta_3)\,d\eta_1 \,d\eta_3.
%\end{align*}
%In other words, we need to compute the Fourier transform of $\delta(2\eta_1 \cdot \eta_3)$, or at least characterize this object (decay, etc)... which I don't know how to do! Do you guys have an idea? If I am not mistaken, in dimension 2, we find that the Fourier transform of $\delta(2\eta_1 \cdot \eta_3)$ is itself.
%}

\noindent {\bf Acknowledgments.} TB was supported by the National Science Foundation  grant DMS-1600868.
PG was supported by the National Science Foundation  grant  DMS-1301380.
JS was  supported by the National Science Foundation  grant DMS-1363013.
ZH was supported by National Science Foundation grant DMS-1600561, a Sloan Fellowship, and a startup fund from Georgia Institute of Technology.

\vspace{.1in}


\begin{thebibliography}{0}

\bibitem{BGHS} Buckmaster, Tristan; Germain, Pierre; Hani, Zaher; Shatah, Jalal; Effective dynamics of the nonlinear Schr\"odinger equation on large domains, preprint.

\bibitem{DHR} Duyckaerts, Thomas; Holmer, Justin; Roudenko, Svetlana; Scattering for the non-radial 3D cubic nonlinear Schr\"odinger equation. \textit{Math. Res. Lett.} 15 (2008), no. 6, 1233–1250.

\bibitem{FGH}  Faou, Erwan; Germain, Pierre; Hani, Zaher; The weakly nonlinear large-box limit of the 2D cubic nonlinear Schrödinger equation. \textit{J. Amer. Math. Soc.} 29 (2016), 915--982.

\bibitem{DF} Foschi, Damiano; Maximizers for the Strichartz inequality. \textit{J. Eur. Math. Soc.} (JEMS) 9 (2007), 739--774.

\bibitem{PG} Germain, Pierre; From dispersion management to continuous resonant equations, \textit{Proceedings of the ICMP 2015}.

\bibitem{GHT1} Germain, Pierre; Hani, Zaher; Thomann, Laurent; On the continuous resonant equation for NLS. I. Deterministic analysis. \textit{J. Math. Pures Appl.} (9) 105 (2016), 131--163.

\bibitem{GHT2} Germain, Pierre; Hani, Zaher; Thomann, Laurent; On the continuous resonant equation for NLS, II: Statistical study. \textit{Anal. PDE} 8 (2015), no. 7, 1733--1756. 

\bibitem{HaniThomann} Hani, Zaher; Thomann, Laurent, Asymptotic behavior of the nonlinear Schr\"odinger equation with harmonic trapping. Communications in Pure and Applied Mathematics (CPAM) Volume 69, Issue 9, September 2016, Pages 1727--1776.

\bibitem{HZ} Hundertmark, Dirk; Zharnitsky, Vadim; On sharp Strichartz inequalities in low dimensions. \textit{Int. Math. Res. Not.} 2006.

\bibitem{K} Keraani, Sahbi; On the defect of compactness for the Strichartz estimates of the Schr\"odinger equation, \textit{J. Diff. Eq.} 175 (2001), pp. 353--392.

\bibitem{SS} Shao, Shuanglin; Maximizers for the Strichartz and the Sobolev-Strichartz inequalities for the Schr\"odinger equation. \textit{Electron. J. Differential Equations} 2009, No. 3.

\bibitem{TT}  Tao, Terence; \textit{Nonlinear dispersive equations. Local and global analysis.} CBMS Regional Conference Series in Mathematics, 106. Published for the Conference Board of the Mathematical Sciences, Washington, DC; by the American Mathematical Society, Providence, RI, 2006.
\end{thebibliography}
\end{document}